\documentclass[12pt]{article}
\usepackage{amssymb}

\newtheorem{theorem}{Theorem}[section]
\newtheorem{proposition}[theorem]{Proposition}

\newtheorem{corollary}[theorem]{Corollary}
\newtheorem{definition}[theorem]{Definition}

\newtheorem{remark}[theorem]{Remark}
\newtheorem{lemma}[theorem]{Lemma}

\newcommand{\mbb}{\mathbb}

\newcommand{\lk}{\mbox{\upshape lk}}
\newcommand{\st}{\overline{\mbox{\upshape st}}}
\newcommand{\p}{^{\prime}}
\newcommand{\pp}{^{\prime\prime}}
\newcommand{\rk}{\mbox{\upshape rk}}
\newcommand{\init}{\mbox{\upshape init}}
\newcommand{\term}{\mbox{\upshape term}}
\newcommand{\Sd}{\mbox{\upshape Sd}}

\title{The Fundamental Group of Balanced Simplicial Complexes and Posets}
\author{Steven Klee\\
\small Department of Mathematics, Box 354350\\[-0.8ex]
\small University of Washington, Seattle, WA 98195-4350, USA,\\[-0.8ex]
\small \texttt{klees@math.washington.edu}}

\begin{document}
\maketitle

\begin{center}
\textit{Dedicated to Anders Bj\"orner on the occasion of his 60th birthday.}
\end{center}

\begin{abstract}
We establish an upper bound on the cardinality of a minimal generating set for the fundamental group of a large family of connected, balanced simplicial complexes and, more generally, simplicial posets.
\end{abstract}

\section{Introduction}
One commonly studied combinatorial invariant of a finite $(d-1)$-dimensional simplicial complex $\Delta$ is its $f$-vector $f = (f_0, \ldots, f_{d-1})$ where $f_i$ denotes the number of $i$-dimensional faces of $\Delta$.  This leads to the study of the $h$-numbers of $\Delta$ defined by the relation $\sum_{i=0}^dh_i\lambda^{d-i} = \sum_{i=0}^df_{i-1}(\lambda-1)^{d-i}$.  A great deal of work has been done to relate the $f$-numbers and $h$-numbers of $\Delta$ to the dimensions of the singular homology groups of $\Delta$ with coefficients in a certain field; see, for example, the work of Bj\"orner and Kalai in  \cite{BK88} and \cite{BK89}, and Chapters 2 and 3 of Stanley \cite{S96}.  In comparison, very little seems to be known about the relationship between the $f$-numbers of a simplicial complex and various invariants of its homotopy groups.  In this paper, we bound the minimal number of generators of the fundamental group of a balanced simplicial complex in terms of $h_2$.  More generally, we bound the minimal number of generators of the fundamental group of a balanced simplicial poset in terms of $h_2$.

It was conjectured by Kalai \cite{K87} and proved by Novik and Swartz in \cite{NS} that if $\Delta$ is a $(d-1)$-dimensional manifold that is orientable over the field $\mathbf{k}$, then $$h_2-h_1 \geq {d+1 \choose 2} \beta_1,$$ where $\beta_1$ is the dimension of the singular homology group $H_1(\Delta;\mathbf{k})$.  The Hurewicz Theorem (see Spanier \cite{Spanier}) says that $H_1(X;\mathbb{Z})$ is isomorphic to the abelianization of $\pi_1(X,*)$ for a connected space $X$.  We will see below that $\pi_1(\Delta,*)$ is finitely generated.  Thus the Hurewicz Theorem says that the minimal number of generators of the fundamental group of a simplicial complex $\Delta$ is greater than or equal to the number of generators of $H_1(\Delta;\mbb{Z})$.  By the universal coefficient theorem, $H_1(\Delta;\mathbf{k}) \approx H_1(\Delta;\mathbb{Z})\otimes \mathbf{k}$ for any field $\mathbf{k}$; and,  consequently, the minimal number of generators of $\pi_1(\Delta,*)$ is greater than or equal to $\beta_1(\Delta)$ for any field $\mathbf{k}$.

In this paper, we study simplicial complexes and simplicial posets $\Delta$ that are pure and balanced with the property that every face $F \in \Delta$ of codimension at least $2$ (including the empty face) has connected link.  This includes the class of balanced triangulations of compact manifolds and, using the language of Goresky and MacPherson in \cite{GM}, the more general class of balanced normal pseudomanifolds.  Under these weaker assumptions, we show that $$h_2 \geq {d \choose 2}m(\Delta),$$ where $m(\Delta)$ denotes the minimal number of generators of $\pi_1(\Delta,*)$.

The paper is structured as follows.  Section 2 contains all necessary definitions and background material.  In Section 3, we outline a sequence of theorems in algebraic topology that are used to give a description of the fundamental group in terms of a finite set of generators and relations.  In Section 4, we use the theorems in Section 3 to prove Theorem \ref{PathEquiv}.  This theorem gives the desired bound on $m(\Delta)$.  In Section 5, after giving some definitions related to simplicial posets, we extend the topological results in Section 3 and the result of Theorem \ref{PathEquiv} to the class of simplicial posets.

\section{Notation and Conventions}
Throughout this paper, we assume that $\Delta$ is a $(d-1)$-dimensional simplicial complex on vertex set $V = \{v_1, \ldots, v_n\}$.  We recall that the dimension of a face $F \in \Delta$ is $\dim F = |F|-1$, and the dimension of $\Delta$ is $\dim \Delta = \max\{\dim F: F \in \Delta\}$.  A simplicial complex is \textit{pure} if all of its facets (maximal faces) have the same dimension.  The \textit{link} of a face $F \in \Delta$ is the subcomplex $$\lk_{\Delta}F = \{G \in \Delta: F \cap G = \emptyset, F \cup G \in \Delta\}.$$  Similarly, the \textit{closed star} of a face $F \in \Delta$ is the subcomplex $$\st_{\Delta}F = \{G \in \Delta: F \cup G \in \Delta\}.$$

The \textit{geometric realization} of $\Delta$, denoted by $|\Delta|$, is the union over all faces $F \in \Delta$ of the convex hull in $\mathbb{R}^n$ of $\{e_i: v_i \in F\}$ where $\{e_1, \ldots, e_n\}$ denotes the standard basis in $\mathbb{R}^n$.  Given this geometric realization, we will make little distinction between the combinatorial object $\Delta$ and the topological space $|\Delta|$.  For example, we will often abuse notation and write $H_i(\Delta;\mathbf{k})$ instead of the more cluttered $H_i(|\Delta|;\mathbf{k})$.

The $f$-\textit{vector} of $\Delta$ is the vector $f=(f_{-1},f_0,f_1,\ldots, f_{d-1})$ where $f_i$ denotes the number of $i$-dimensional faces of $\Delta$.  By convention, we set $f_{-1}=1$, corresponding to the empty face.  If it is important to distinguish the simplicial complex $\Delta$, we write $f(\Delta)$ for the $f$-vector of $\Delta$, and $f_i(\Delta)$ for its $f$-numbers (i.e. the entries of its $f$-vector).  Another important combinatorial invariant of $\Delta$ is the $h$-\textit{vector} $h=(h_0,\ldots, h_d)$ where
\begin{displaymath}
h_i = \sum_{j=0}^i (-1)^{i-j}{d-j \choose d-i}f_{j-1}.
\end{displaymath}

For us, it will be particularly important to study a certain class of complexes known as balanced simplicial complexes, which were introduced by Stanley in \cite{S79}.

\begin{definition}
A $(d-1)$-dimensional simplicial complex $\Delta$ is \textbf{balanced} if its $1$-skeleton, considered as a graph, is $d$-colorable.  That is to say there is a coloring $\kappa: V \rightarrow [d]$ such that for all $F \in \Delta$ and distinct $v,w \in F$, we have $\kappa(v) \neq \kappa(w)$.  We assume that a balanced complex $\Delta$ comes equipped with such a coloring $\kappa$.
\end{definition}

The order complex of a rank-$d$ graded poset is one example of a balanced simplicial complex.  If $\Delta$ is a balanced complex and $S \subseteq [d]$, it is often important to study the \textit{$S$-rank selected subcomplex} of $\Delta$, which is defined as
\begin{displaymath}
\Delta_S = \{F \in \Delta: \kappa(F) \subseteq S\};
\end{displaymath}
that is, for a fixed coloring $\kappa$, we define $\Delta_S$ to be the subcomplex of faces whose vertices are colored with colors from $S$.  In \cite{S79} Stanley showed that
\begin{equation}\label{hnums}
h_i(\Delta) = \sum_{|S|=i}h_i(\Delta_S).
\end{equation}

\section{The Edge-Path Group}
In order to obtain a concrete description of $\pi_1(\Delta,*)$ that relies only on the structure of $\Delta$ as a simplicial complex, we introduce the \textit{edge-path group} of $\Delta$ (see, for example, Seifert and Threlfall \cite{ST80} or Spanier \cite{Spanier}).  This will ultimately allow us to relate the combinatorial data of $f(\Delta)$ to the fundamental group of $\Delta$.

An \textit{edge} in $\Delta$ is an ordered pair of vertices $(v,v^{\prime})$ with $\{v,v^{\prime}\} \in \Delta$.  An \textit{edge path} $\gamma$ in $\Delta$ is a finite nonempty sequence $(v_0,v_1)(v_1,v_2)\cdots(v_{r-1},v_r)$ of edges in $\Delta$. We say that $\gamma$ is an edge path from $v_0$ to $v_r$, or that $\gamma$ starts at $v_0$ and ends at $v_r$.  A \textit{closed} edge path at $v$ is an edge path $\gamma$ such that $v_0=v=v_r$.

We say that two edge paths $\gamma$ and $\gamma^{\prime}$ are \textit{simply equivalent} if there exist vertices $v,v^{\prime},v^{\prime\prime}$ in $\Delta$ with $\{v,v^{\prime},v^{\prime\prime}\} \in \Delta$ such that the unordered pair $\{\gamma,\gamma^{\prime}\}$ is equal to one of the following unordered pairs:
\begin{itemize}
  \item $\{(v,v\pp),(v,v\p)(v\p,v\pp)\}$,
  \item $\{\gamma_1(v,v\pp),\gamma_1(v,v\p)(v\p,v\pp)\}$ for some edge path $\gamma_1$ ending at $v$,
  \item $\{(v,v\pp)\gamma_2,(v,v\p)(v\p,v\pp)\gamma_2\}$ for some edge path $\gamma_2$ starting at $v\pp$,
  \item $\{\gamma_1(v,v\pp)\gamma_2,\gamma_1(v,v\p)(v\p,v\pp)\gamma_2\}$ for edge paths $\gamma_1,\gamma_2$ as above.
\end{itemize}

We note that the given vertices $v,v\p,v\pp \in \Delta$ need not be distinct.  For example, $(v,v)$ is a valid edge (the edge that does not leave the vertex $v$), and we have the simple equivalence $(v,v\p)(v\p,v) \sim (v,v)$.  We say that two edge paths $\gamma$ and $\gamma\p$ are \textit{equivalent}, and write $\gamma \sim \gamma\p$, if there is a finite sequence of edge paths $\gamma_0, \gamma_1, \ldots, \gamma_s$ such that $\gamma = \gamma_0$, $\gamma\p = \gamma_s$ and $\gamma_i$ is simply equivalent to $\gamma_{i+1}$ for $0 \leq i \leq s-1$.  It is easy to check that this defines an equivalence relation on the collection of edge paths $\gamma$ in $\Delta$ starting at $v$ and ending at $v\p$.  Moreover, for two edge paths $\gamma$ and $\gamma\p$ with the terminal vertex of $\gamma$ equal to the initial vertex of $\gamma\p$, we can form their product edge path $\gamma\gamma\p$ by concatenation.

Now we pick a base vertex $v_0 \in \Delta$.  Let $E(\Delta,v_0)$ denote the set of equivalence classes of closed edge paths in $\Delta$ based at $v_0$. We multiply equivalence classes by $[\gamma]*[\gamma\p] = [\gamma\gamma\p]$ to give $E(\Delta,v_0)$ a group structure called the \textit{edge path group} of $\Delta$ based at $v_0$.

The Cellular Approximation Theorem (\cite{Spanier} VII.6.17)  tells us that any path in $\Delta$ is homotopic to a path in the $1$-skeleton of $\Delta$.  We use this fact to motivate the proof of the following theorem from Spanier.

\begin{theorem}[\cite{Spanier} III.6.17]
If $\Delta$ is a simplicial complex and $v_0 \in \Delta$, then there is a natural isomorphism
\begin{displaymath}
E(\Delta,v_0) \approx \pi_1(\Delta,v_0).
\end{displaymath}
\end{theorem}

For a connected simplicial complex $\Delta$ we will also consider the group $G$, defined as follows.  Let $T$ be a spanning tree in the $1$-skeleton of $\Delta$.  Since $\Delta$ is connected, such a spanning tree exists.  We define $G$ to be the free group generated by edges $(v,v^{\prime}) \in \Delta$ modulo the relations
\renewcommand\theenumi{[R\arabic{enumi}]}
\begin{enumerate}
\item $(v,v^{\prime}) = 1$ if $(v,v^{\prime}) \in T$, and
\item $(v,v^{\prime})(v^{\prime},v^{\prime\prime}) = (v,v^{\prime\prime})$ if $\{v,v^{\prime},v^{\prime\prime}\} \in \Delta$.
\end{enumerate}

The following theorem will be crucial in our study of the fundamental group.
\begin{theorem}[\cite{Spanier} III.7.3]\label{Gpi1}
With the above notation, $$E(\Delta,v_0) \approx G.$$
\end{theorem}

\noindent We note for later use that this isomorphism is given by the map $$\Phi: E(\Delta,v_0) \rightarrow G$$ that sends $[(v_0,v_1)(v_1,v_2)\cdots(v_{r-1},v_r)]_E \mapsto [(v_0,v_1)(v_1,v_2)\cdots(v_{r-1},v_r)]_G$.  Here, $[-]_E$ and $[-]_G$ denote the equivalence classes of an edge path in $E(\Delta,v_0)$ and $G$, respectively.  The inverse to this map is defined on the generators of $G$ as follows.  For $(v,v\p) \in \Delta$, there is an edge path $\gamma$ from $v_0$ to $v$ along $T$ and an edge path $\gamma\p$ from $v\p$ to $v_0$ along $T$.  Using these paths, we map $\Phi^{-1}[(v,v\p)]_G = [\gamma(v,v\p)\gamma\p]_E.$

\section{The Fundamental Group and $h$-numbers} \label{fundgroup}
Our goal now is to use Theorem \ref{Gpi1} to bound the minimal number of generators of $\pi_1(\Delta,*)$.  For ease of notation, let $m(\Delta,*)$ denote the minimal number of generators of $\pi_1(\Delta,*)$. When the basepoint is understood or irrelevant (e.g. when $\Delta$ is connected) we will write $m(\Delta)$ in place of $m(\Delta,*)$.  For the remainder of this section, we will be concerned with simplicial complexes $\Delta$ of dimension $(d-1)$ with the following properties:
\renewcommand\theenumi{(\Roman{enumi})}
\begin{enumerate}
  \item $\Delta$ is pure,
  \item $\Delta$ is balanced,
  \item $\lk_{\Delta} F$ is connected for all faces $F \in \Delta$ with $0 \leq |F| < d-1$.
\end{enumerate}
\noindent In particular, property (III) implies that $\Delta$ is connected by taking $F$ to be the empty face.

Since results on balanced simplicial complexes are well-suited to proofs by induction, we begin with the following observation.
\begin{proposition}\label{PropInduction}
  Let $\Delta$ be a simplicial complex with $d \geq 2$ that satisfies properties (I)--(III).  If $F \in \Delta$ is a face with $|F|<d-1$, then $\lk_{\Delta}F$ satisfies properties (I)--(III) as well.
\end{proposition}
\begin{proof}
When $d=2$, the result holds trivially since the only such face $F$ is the empty face.  When $d>3$ and $F$ is nonempty, it is sufficient to show that the result holds for a single vertex $v \in F$.  Indeed, if we set $G = F \setminus \{v\}$, then $\lk_{\Delta}F = \lk_{\lk_{\Delta}v}G$ at which point we may appeal to induction on $|F|$.

We immediately see that $\lk_{\Delta}v$ inherits properties (I) and (II) from $\Delta$.  Finally, if $\sigma \in\lk_{\Delta}v$ is a face with $|\sigma|<d-2$, then $\lk_{\lk_{\Delta}v}\sigma = \lk_{\Delta}(\sigma \cup v)$ is connected by property (III).
\end{proof}

\begin{lemma}\label{lem1}
Let $\Delta$ be a $(d-1)$-dimensional simplicial complex with $d \geq 2$ that satisfies properties (I) and (III).  If  $F$ and $F^{\prime}$ are facets in $\Delta$, then there is a chain of facets $$F=F_0, F_1, \ldots, F_m = F^{\prime} \qquad \qquad (*)$$ such that $|F_i \cap F_{i+1}| = d-1$ for all $i$.
\end{lemma}
\begin{remark}
We say that a pure simplicial complex satisfying property (*) is \textbf{strongly connected}.
\end{remark}
\begin{proof}
We proceed by induction on $d$.  When $d=2$, $\Delta$ is a connected graph, and such a chain of facets is a path from some vertex $v \in F$ to a vertex $v\p \in F\p$.  We now assume that $d \geq 3$.

First, we note that the closed star of each face in $\Delta$ is strongly connected.  Indeed, by induction the link (and hence the closed star $\st_{\Delta}\sigma$) of each face $\sigma \in \Delta$ with $|\sigma|<d-1$ is strongly connected.  On the other hand, if $\sigma \in \Delta$ is a face with $|\sigma| = d-1$, then every facet in $\st_{\Delta}\sigma$ contains $\sigma$ and so $\st_{\Delta}\sigma$ is strongly connected as well.  Finally, if $\sigma$ is a facet, then $\st_{\Delta}\sigma$ is strongly connected as it only contains a single facet.

It is also clear that if $\Delta\p$ and $\Delta\pp$ are strongly connected subcomplexes of $\Delta$ such that $\Delta\p \cap \Delta\pp$ contains a facet, then $\Delta\p \cup \Delta\pp$ is strongly connected as well.  Finally, suppose $\Delta_0 \subseteq \Delta$ is a maximal strongly connected subcomplex of $\Delta$.  If $F \in \Delta_0$ is any face, then $\st_{\Delta}F$ intersects $\Delta_0$ in a facet.  Since $\st_{\Delta}F \cup \Delta_0$ is strongly connected and $\Delta_0$ is maximal, we must have $\st_{\Delta}F \subseteq \Delta_0$.  Thus $\Delta_0$ is a connected component of $\Delta$.  Since $\Delta$ is connected, $\Delta = \Delta_0$.
\end{proof}

\begin{lemma}\label{lem2}
Let $\Delta$ be a $(d-1)$-dimensional simplicial complex with $d \geq 2$ that satisfies properties (I)--(III).  For any $S \subseteq [d]$ with $|S|=2$, the rank selected subcomplex $\Delta_S$ is connected.
\end{lemma}
\begin{proof}
Say $S = \{c_1,c_2\}$.  Pick vertices $v, v^{\prime} \in \Delta_S$ and facets $F \ni v$, $F^{\prime} \ni v^{\prime}$.  By Lemma \ref{lem1}, there is a chain of facets $F = F_1, \ldots, F_m = F^{\prime}$ for which $F_i$ intersects $F_{i+1}$ in a codimension $1$ face.  We claim that a path from $v$ to $v^{\prime}$ in $\Delta_S$ can be found in $\cup_{i=1}^m F_i$.

When $m=1$, $\{v,v^{\prime}\}$ is an edge in $F = F_1 = F^{\prime}$.  For $m>1$, we examine the facet $F_1$.  Without loss of generality, say $\kappa(v) = c_1$, and let $w \in F_1$ be the vertex with $\kappa(w) = c_2$.  If $\{v,w\} \in F_1 \cap F_2$, then the facet $F_1$ in our chain is extraneous, and we could have taken $F = F_2$ instead.  Inductively, we can find a path from $v$ to $v^{\prime}$ in $\Delta_S$ that is contained in $\cup_{i=2}^mF_i$.  On the other hand, if $v \notin F_2$, then we can find a path from $w$ to $v^{\prime}$ in $\Delta_S$ that is contained in $\cup_{i=2}^mF_i$ by induction.  Since $(v,w) \in \Delta_S$, this path extends to a path from $v$ to $v^{\prime}$ in $\Delta_S$.
\end{proof}

\begin{theorem}\label{PathEquiv}
Let $\Delta$ be a $(d-1)$-dimensional simplicial complex with $d \geq 2$ that satisfies properties (I)--(III), and $S \subseteq [d]$ with $|S|=2$.  If $v,v^{\prime}$ are vertices in $\Delta_S$, then any edge path $\gamma$ from $v$ to $v^{\prime}$ in $\Delta$ is equivalent to an edge path from $v$ to $v^{\prime}$ in $\Delta_S$.
\end{theorem}
\begin{proof}
When $d=2$, $\Delta_S = \Delta$, and the result holds trivially, so we can assume $d \geq 3$.

We may write our edge path $\gamma$ as a sequence $$\gamma = (v_0,v_1)(v_1,v_2)\cdots(v_{r-1},v_r)$$ where $v_0=v$, $v_r=v^{\prime}$, and $\{v_i,v_{i+1}\} \in \Delta$ for all $i$. We establish the claim by induction on $r$.  When $r=1$, the edge $(v,v^{\prime})$ is already an edge in $\Delta_S$.  Now we assume $r>1$.  If $v_1 \in \Delta_S$, the sequence $(v_1,v_2)\cdots(v_{r-1},v_r)$ is equivalent to an edge path $\widetilde{\gamma}$ from $v_1$ to $v\p$ in $\Delta_S$ by our induction hypothesis on $r$. Hence $\gamma$ is equivalent to $(v_0,v_1)\widetilde{\gamma}$.

On the other hand, suppose that $v_1 \notin \Delta_S$.  Since $\kappa(v_1) \notin S$ and $\Delta$ is pure and balanced, there is a vertex $\widetilde{v}\in \Delta_S$ such that $\{v_1,v_2,\widetilde{v}\} \in \Delta$.  By Proposition~\ref{PropInduction}, $\lk_{\Delta}v_1$ is a simplicial complex of dimension at least $1$ satisfying properties (I)--(III).  Thus by Lemma \ref{lem2}, there is an edge path $\gamma\p = (u_0,u_1)\cdots(u_{k-1},u_k)$ such that $u_0 = v_0$, $u_k = \widetilde{v}$, and each edge $\{u_i,u_{i+1}\} \in (\lk_{\Delta}v_1)_S$.  Since each edge $\{u_i,u_{i+1}\} \in \lk_{\Delta}v_1$, it follows that $\{u_i,u_{i+1},v_1\} \in \Delta$ for all $i$.

We now use the fact that $(u,u\p)(u\p,u\pp) \sim (u,u\pp)$ for all $\{u,u\p,u\pp\} \in \Delta$ to see the following simple equivalences of edge paths.

\begin{eqnarray*}
(v_0,v_1)(v_1,\widetilde{v}) &=& (u_0,v_1)(v_1,\widetilde{v}) \\
&\sim& (u_0,u_1)(u_1,v_1)(v_1,\widetilde{v}) \\
&\sim& (u_0,u_1)(u_1,u_2)(u_2,v_1)(v_1,\widetilde{v}) \\
&& \ldots \\
&\sim& (u_0,u_1)(u_1,u_2)\cdots (u_{k-2},u_{k-1})(u_{k-1},v_1)(v_1,\widetilde{v}) \\
&\sim& (u_0,u_1)(u_1,u_2)\cdots (u_{k-2},u_{k-1})(u_{k-1},\widetilde{v}).
\end{eqnarray*}

For convenience, we write $\gamma_1 = (u_0,u_1)(u_1,u_2)\cdots (u_{k-2},u_{k-1})(u_{k-1},\widetilde{v})$.  Now we observe that $(v_0,v_1)(v_1,v_2) \sim (v_0,v_1)(v_1,\widetilde{v})(\widetilde{v},v_2)$ so that
\begin{eqnarray*}
  \gamma &=& (v_0,v_1)(v_1,v_2)(v_2,v_3)\cdots(v_{r-1},v_r) \\
  &\sim& (v_0,v_1)(v_1,\widetilde{v})(\widetilde{v},v_2)(v_2,v_3)\cdots(v_{r-1},v_r) \\
  &\sim& \gamma_1(\widetilde{v},v_2)(v_2,v_3)\cdots(v_{r-1},v_r).
\end{eqnarray*}
By induction on $r$, there is an edge path $\gamma_2$ in $\Delta_S$ from $\widetilde{v}$ to $v_r$ that is equivalent to $(\widetilde{v},v_2)(v_2,v_3)\cdots(v_{r-1},v_r)$ so that $\gamma \sim \gamma_1\gamma_2.$  Thus, indeed, $\gamma$ is equivalent to an edge path in $\Delta_S$.
\end{proof}

Setting $v = v\p = v_0$, we have the following corollary.

\begin{corollary}\label{Egen}
If $v_0 \in \Delta_S$, every class in $E(\Delta, v_0)$ can be represented by a closed edge path in $\Delta_S$.
\end{corollary}

Now we have an explicit description of a smaller generating set of $\pi_1(\Delta,v_0)$.

\begin{lemma}\label{DeltaSGens}
Let $\Delta$ be a $(d-1)$-dimensional simplicial complex with $d \geq 2$ that satisfies properties (I)--(III).  For a fixed $S \subseteq [d]$ with $|S|=2$, the group $G$ of Theorem \ref{Gpi1} is generated by the edges $(v,v\p)$ with $\{v,v\p\} \in \Delta_S$.
\end{lemma}
\begin{proof}
In order to use Theorem \ref{Gpi1}, we must choose some spanning tree $T$ in the $1$-skeleton of $\Delta$.  We will do this in a specific way.  Since $\Delta_S$ is a connected graph, we can find a spanning tree $\widetilde{T}$ in $\Delta_S$.  Since $\Delta$ is connected, we can extend $\widetilde{T}$ to a spanning tree $T$ in $\Delta$ so that $\widetilde{T} \subseteq T$.

By Corollary \ref{Egen}, each class in $E(\Delta,v_0)$ is represented by a closed edge path in $\Delta_S$, and hence the isomorphism $\Phi$ of Theorem \ref{Gpi1} maps $E(\Delta,v_0)$ into the subgroup of $H \subseteq G$ generated by edges $(v,v\p) \in \Delta_S$.  Since $\Phi$ is surjective, we must have $H=G$.

\end{proof}

\begin{corollary}
With $\Delta$ and $S$ as in Lemma \ref{DeltaSGens}, we have $$m(\Delta) \leq h_2(\Delta_S).$$
\end{corollary}
\begin{proof}
Lemma \ref{DeltaSGens} tells us that the $f_1(\Delta_S)$ edges in $\Delta_S$ generate the group $G$.  Since our spanning tree $T$ contains a spanning tree in $\Delta_S$, $f_0(\Delta_S)-1$ of these generators will be identified with the identity.  Thus $$m(\Delta) \leq f_1(\Delta_S) - f_0(\Delta_S)+1 = h_2(\Delta_S).$$
\end{proof}

While the proof of the above corollary requires specific information about the set $S$ and a specific spanning tree $T \subset \Delta$, its result is purely combinatorial.  Since $\Delta$ is connected, $\pi_1(\Delta,*)$ is independent of the basepoint, and so we can sum over all such sets $S \subset [d]$ with $|S|=2$ to get
\begin{eqnarray*}
{d \choose 2}m(\Delta) &\leq& \sum_{|S|=2} h_2(\Delta_S) \\
&=& h_2(\Delta) \qquad \mbox{by Equation (\ref{hnums})}.
\end{eqnarray*}

This gives the following theorem.
\begin{theorem} \label{boundGens}
Let $\Delta$ be a pure, balanced simplicial complex of dimension $(d-1)$ with the property that $\lk_{\Delta}F$ is connected for all faces $F \in \Delta$ with $|F| < d-1$.  Then $${d \choose 2}m(\Delta) \leq h_2(\Delta).$$
\end{theorem}

\section{Extensions and Further Questions}
\subsection{Simplicial Posets}
We now generalize the results in Section \ref{fundgroup} to the class of simplicial posets.  A \textit{simplicial poset} is a poset $P$ with a least element $\hat{0}$ such that for any $x \in P \setminus \{\hat{0}\}$, the interval $[\hat{0},x]$ is a Boolean algebra (see  Bj\"orner \cite{B84} or  Stanley \cite{S91}).  That is to say that the interval $[\hat{0},x]$ is isomorphic to the face poset of a simplex.  Thus $P$ is graded by $\rk(\sigma) = k+1$ if $[\hat{0},\sigma]$ is isomorphic to the face poset of a $k$-simplex.  The face poset of a simplicial complex is a simplicial poset.  Following \cite{B84}, we see that every simplicial poset $P$ has a geometric interpretation as the face poset of a regular CW-complex $|P|$ in which each cell is a simplex and each pair of simplices is joined along a possibly empty subcomplex of their boundaries.  We call $|P|$ the \textit{realization} of $P$.  With this geometric picture in mind, we refer to elements of $P$ as \textit{faces} and work interchangeably between $P$ and $|P|$.  In particular, we refer to rank-$1$ elements of $P$ as vertices and maximal rank elements of $P$ as facets.  As in the case of simplicial complexes, we say that the \textit{dimension} of a face $\sigma \in P$ is $\rk(\sigma) -1$, and the dimension of $P$ is $d-1$ where $d = \rk (P) = \max\{\rk(\sigma): \sigma \in P\}$.  We say that $P$ is \textit{pure} if each of its facets has the same rank.  In addition, we can form the order complex $\Delta(\overline{P})$ of the poset $\overline{P} = P \setminus \{\hat{0}\}$, which gives a barycentric subdivision of $|P|$.

As with simplicial complexes, we define the link of a face $\tau \in P$ as $$\lk_P\tau = \{\sigma \in P : \sigma \geq \tau\}.$$  \noindent It is worth noting that $\lk_P\tau$ is a simplicial poset whose minimal element is $\tau$, but $\lk_P\tau$ is not necessarily a subcomplex of $|P|$.  All hope is not lost, however, since for any saturated chain $F = \{\tau_0<\tau_1<\ldots<\tau_r=\tau\}$ in $(\hat{0},\tau]$ we have $\lk_{\Delta(\overline{P})}(F) \cong \Delta(\overline{\lk_P(\tau)})$.  Here we say $F$ is \textit{saturated} if each relation $\tau_i < \tau_{i+1}$ is a covering relation in $P$.

We are also concerned with balanced simplicial posets and strongly connected simplicial posets.  Suppose $P$ is a pure simplicial poset of dimension $(d-1)$, and let $V$ denote the vertex set of $P$. We say that $P$ is \textit{balanced} if there is a coloring $\kappa: V \rightarrow [d]$ such that for each facet $\sigma \in P$ and distinct vertices $v,w < \sigma$, we have $\kappa(v) \neq \kappa(w)$.  If $S \subseteq [d]$, we can form the $S$-rank selected poset of $P$, defined as $$P_S = \{\sigma \in P: \kappa(\sigma) \subseteq S\} \mbox{  where  } \kappa(\sigma) = \{\kappa(v): v < \sigma, \rk(v) = 1\}.$$  We say that $P$ is \textit{strongly connected} if for all facets $\sigma,\sigma\p \in P$ there is a chain of facets $$\sigma = \sigma_0,\sigma_1,\ldots, \sigma_m = \sigma\p,$$ and faces $\tau_i$ of rank $d-1$ such that $\tau_i$ is covered by $\sigma_i$ and $\sigma_{i+1}$ for all $0 \leq i \leq m-1$.  For simplicial complexes, the face $\tau_i$ is naturally $\sigma_i \cap \sigma_{i+1}$; however, for simplicial posets, the face $\tau_i$ is not necessarily unique.

As in Section 4, we are concerned with simplicial posets $P$ of rank $d$ satisfying the following three properties:
\renewcommand\theenumi{(\roman{enumi})}
\begin{enumerate}
  \item $P$ is pure,
  \item $P$ is balanced,
  \item $\lk_{P} \sigma$ is connected for all faces $\sigma \in P$ with $0 \leq \rk(\sigma) < d-1$.
\end{enumerate}

Our first task is to understand the fundamental group of a simplicial poset by constructing an analogue of the edge-path group of a simplicial complex.  We have to be careful because there can be several edges connecting a given pair of vertices.  An \textit{edge} in $P$ is an oriented rank-2 element $e \in P$ with an initial vertex, denoted $\init(e)$, and a terminal vertex, denoted $\term(e)$.  If $e$ is an edge, we let $e^{-1}$ denote its \textit{inverse edge}, that is, we interchange the initial and terminal vertices of $e$, reversing the orientation of $e$.  We note that the initial and terminal vertices of $e$ are distinct since $[\hat{0},e]$ is a Boolean algebra.  We also allow for the degenerate edge $e = (v,v)$ for any vertex $v \in P$.  An \textit{edge path} $\gamma$ in $P$ is a finite nonempty sequence $e_0e_1\cdots e_r$ of edges in $P$ such that $\term(e_i) = \init(e_{i+1})$ for all $0 \leq i \leq r-1$.  A \textit{closed} edge path at $v$ is an edge path $\gamma$ such that $\init(e_0) = v = \term(e_r)$.  Given edge paths $\gamma$ from $v$ to $v\p$ and $\gamma\p$  from $v\p$ to $v\pp$, we can form their product edge path $\gamma\gamma\p$ from $v$ to $v\pp$ by concatenation.

Suppose $\sigma \in P$ is a rank-3 face with (distinct) vertices $v,v\p$ and $v\pp$ and edges $e,e\p$ and $e\pp$ with $\init(e) = v = \init(e\pp)$, $\init(e\p) = v\p = \term(v)$ and $\term(e\pp) = v\pp = \term(e\p)$.  Analogously to Section 3, we say that two edge paths $\gamma$ and $\gamma\p$ are \textit{simply equivalent} if the unordered pair $\{\gamma,\gamma\p\}$ is equal to one of the following unordered pairs:

\begin{itemize}
  \item $\{e\pp, ee\p\}$ or $\{(v,v), ee^{-1}\}$;
  \item $\{\gamma_1e\pp,\gamma_1ee\p\}$ or
  $\{\gamma_1, \gamma_1ee^{-1}\}$ for some edge path $\gamma_1$ ending at $v$;
  \item $\{e\pp\gamma_2,ee\p\gamma_2\}$ or
  $\{\gamma_2, (e\p)^{-1}e\p\gamma_2\}$ for some edge path $\gamma_2$ starting at~$v\pp$;
  \item $\{\gamma_1e\pp\gamma_2,\gamma_1ee\p\gamma_2\}$ for edge paths $\gamma_1,\gamma_2$ as above.
\end{itemize}

We say that two edge paths $\gamma$ and $\gamma\p$ are equivalent and write $\gamma \sim \gamma\p$ if there is a finite sequence of edge paths $\gamma = \gamma_0, \ldots, \gamma_s = \gamma\p$ such that $\gamma_i$ is simply equivalent to $\gamma_{i+1}$ for all $i$.  As in the case of simplicial complexes, this forms an equivalence relation on the collection of edge paths in $P$ with initial vertex $v$ and terminal vertex $v\p$.  We pick a base vertex $v_0$ and let $\widetilde{E}(P,v_0)$ denote the collection of equivalence classes of closed edge paths in $P$ at $v_0$.  We give $\widetilde{E}(P,v_0)$ a group structure by loop multiplication, and the resulting group is called the \textit{edge path group} of $P$ based at $v_0$.

Now we ask if the groups $\pi_1(P,v_0)$ and $\widetilde{E}(P,v_0)$ are isomorphic.  As topological spaces, $|P|$ and $\Delta(\overline{P})$ are homeomorphic and so their fundamental groups are isomorphic.  The latter space is a simplicial complex, and so we know that $E(\Delta(\overline{P}),v_0) \approx \pi_1(P,v_0)$.  The following theorem will show that indeed $\pi_1(P,v_0) \approx \widetilde{E}(P,v_0).$

\begin{theorem} \label{edgePathPoset}
Let $P$ be a simplicial poset of rank $d$ satisfying properties $(i)$ and $(iii)$.  If $v_0$ is a vertex in $P$, then $$\widetilde{E}(P,v_0) \approx E(\Delta(\overline{P}),v_0).$$
\end{theorem}
\begin{proof}
Given an edge $e \in P$ with initial vertex $v$ and terminal vertex $v\p$, we define an edge path in $\Delta(\overline{P})$ from $v$ to $v\p$ by barycentric subdivision as $\Sd(e) = (v,e)(e,v\p)$.  We define $\Phi: \widetilde{E}(P,v_0) \rightarrow E(\Delta(\overline{P}),v_0)$ by $$\Phi([e_0e_1\cdots e_r]_{\widetilde{E}}) = [\Sd(e_0)\Sd(e_1)\cdots \Sd(e_r)]_E.$$ \noindent It is easy to check that $\Phi$ is well-defined, as it respects simple equivalences.

We now claim that $\Delta(\overline{P})$ in fact satisfies properties (I)--(III) of Section 4.  Since $\Delta(\overline{P})$ is the order complex of a pure poset, it is pure and balanced.  Indeed, the vertices in $\Delta(\overline{P})$ are elements $\sigma \in \overline{P}$, colored by their rank in $P$.  Finally, for a saturated chain $F = \{\tau_1<\tau_2<\ldots<\tau_r=\tau\}$ in $\overline{P}$ for which $r < d-1$, we see that $\lk_{\Delta(\overline{P})}F \cong \Delta(\overline{\lk_P(\tau)})$ is connected since $\lk_P\tau$ is connected.  By Proposition 3.3 in \cite{D97}, we need only consider saturated chains here.  By Theorem \ref{DeltaSGens}, it follows that any class in $E(\Delta(\overline{P}),v_0)$ can be represented by a closed edge path in $(\Delta(\overline{P}))_{\{1,2\}}$.  In particular, we can represent any class in $E(\Delta(\overline{P}),v_0)$ by an edge path $\gamma = \Sd(e_0)\Sd(e_1)\cdots\Sd(e_r)$ for some edge path $e_0e_1\cdots e_r$ in $P$.  This gives a well-defined inverse to $\Phi$.
\end{proof}


With Theorem \ref{edgePathPoset} and the above definitions, the proofs of Proposition \ref{PropInduction}, Lemmas~\ref{lem1} and \ref{lem2}, and Theorem~\ref{PathEquiv} carry over almost verbatim to the context of simplicial posets and can be used to prove the following Lemma.

\begin{lemma}
Let $P$ be a simplicial poset of rank $d \geq 2$ that satisfies properties (i)--(iii).
\renewcommand\theenumi{\alph{enumi}}
\begin{enumerate}
\item If $\sigma \in P$ is a face and $\rk (\sigma) < d-1$, then $\lk_P \sigma$ satisfies properties (i)--(iii) as well.
\item $P$ is strongly connected.
\item For any $S \subseteq [d]$ with $|S|=2$, the rank selected subcomplex $P_S$ is connected.
\item If $v$ and $v\p$ are vertices in $P_S$, then any edge path $\gamma$ from $v$ to $v\p$ in $P$ is equivalent to an edge path from $v$ to $v\p$ in $P_S$.
\end{enumerate}
\end{lemma}

As in Section 4, part (d) of this Lemma implies the following generalization of Theorem~\ref{boundGens}.

\begin{theorem} \label{boundGens2}
Let $P$ be a pure, balanced simplicial poset of rank $d$ with the property that $\lk_P\sigma$ is connected for each face $\sigma \in P$ with $\rk(\sigma) < d-1$.  Then $${d \choose 2}m(P) \leq h_2(P).$$
\end{theorem}

\subsection{How Tight are the Bounds?}

We now turn our attention to a number of examples to determine if the bounds given by Theorems \ref{boundGens} and \ref{boundGens2} are tight.  We begin by studying a family of simplicial posets constructed by Novik and Swartz in \cite{NS}.  Lemma 7.6 in \cite{NS} constructs a simplicial poset $X(1,d)$ of dimension $(d-1)$ satisfying properties (i)--(iii) whose geometric realization is a $(d-2)$-disk bundle over $\mathbb{S}^1$ and $h_2(X(1,d))={d \choose 2}$.  As $X$ is a bundle over $\mathbb{S}^1$ with contractible fiber, we have $\pi_1(X(1,d),*) \approx \mathbb{Z}$ so that $m(X(1,d))=1$.  This construction shows that the bound in Theorem \ref{boundGens2} is tight.  Moreover, taking connected sums of $r$ copies of $X(1,d)$ (when $d \geq 4$) gives a simplicial poset $P$ whose fundamental group is isomorphic to $\mathbb{Z}^r$ and $h_2(P) = r{d \choose 2}$.  We do not know, however, if the bound in Theorem \ref{boundGens2} is tight when $\pi_1(P,*)$ is either non-free or non-Abelian.  We would also like to know if Theorem \ref{boundGens2} holds if we drop the condition that $P$ is balanced.

\section{Acknowledgements}
I am grateful to my advisor, Isabella Novik, for her guidance, and for carefully editing many preliminary drafts of this paper.  I am also grateful to the anonymous referees who provided many helpful suggestions.

\end{document}